\documentclass{amsart}

\numberwithin{equation}{section}

\usepackage{amssymb}
\usepackage{enumerate, xspace}
\usepackage[colorlinks]{hyperref}

\newtheorem{thm}{Theorem}[section]
\newtheorem{lem}[thm]{Lemma}

\newtheorem{rem}[thm]{Remark}

\newcommand\cK{{\mathcal K}}

\newcommand\cO{{\mathcal O}}

\newcommand\bC{{\mathbb C}}

\newcommand\bE{{\mathbb E}}
\newcommand\bL{{\mathbb L}}
\newcommand\bN{{\mathbb N}}
\newcommand\bP{{\mathbb P}}

\newcommand\bR{{\mathbb R}}

\newcommand\ve{\varepsilon}

\newcommand\vf{\varphi}
\newcommand\Prob{\bP}

\newcommand{\BV}{\mathrm{BV}}
\newcommand{\Var}{\operatorname{Var}}
\newcommand\Id{\mathbb{I}}
\newcommand{\sign}{\operatorname{sign}}
\newcommand{\meta}{\operatorname{meta}}

\begin{document}

\title[Rare events, escape rates and quasistationarity]
{Rare events, escape rates and quasistationarity:\\ some exact formulae}
\author{Gerhard Keller and Carlangelo Liverani}
\address{Gerhard Keller\\Department Mathematik\\Universit\"at
  Erlangen-N\"urnberg\\Bismarckstr. 1$\frac12$, 91052 Erlangen, Germany}
\email{{\tt keller@mi.uni-erlangen.de}}
\address{Carlangelo Liverani\\
Dipartimento di Matematica\\
II Universit\`{a} di Roma (Tor Vergata)\\
Via della Ricerca Scientifica, 00133 Roma, Italy.}
\email{{\tt liverani@mat.uniroma2.it}}
\date{\today}
\begin{abstract}
We present a common framework to study decay and exchanges rates in a wide class of dynamical systems. Several applications, ranging form the metric theory of continuons fractions and the Shannon capacity of contrained systems to the decay rate of metastable states, are given.
\end{abstract}
\thanks{It is a pleasure to thank L.Bunimovich for calling our attention to
  formula \eqref{eq:doubling} in the special case of the doubling map,
  and V.Baladi and M.Demers for discussions on this formula. We are indebted to the ESI where, during the Workshop on Hyperbolic Dynamical Systems with Singularities (2008), this work was started. L.C. thanks the ENS, Paris, where he was invited during part of this work. Also we like to thank the Institut Henri Poincare - Centre Emile Borel where, during the trimester M\'ecanique statistique, probabilit\'es et syst\`emes de particules (2008), this work was finished. Finally, G.K. acknowledges the support by a grant from the DFG}
\maketitle
\section{Introduction}\label{sec:intro}
In applications of the theory of dynamical systems to concrete situations it
is often necessary to study rare events.  Examples are open systems with a
small chance to escape and metastable states.  Although
much numerical work exists (e.g. see \cite{BD, DH} and references therein) not
many rigorous results are available. In principle one can try to apply
perturbation theory but the existing theorems \cite{keller-liverani-1999, GL}
do not produce very sharp results. A similar situation occurs in the study
of linear response theory. While perturbation theory applies to a wide class
of smooth systems \cite{ruelle-1997}, this is no longer true when discontinuities are present in the
system. In that case not only perturbation theory does not imply linear response, but
in fact there are cases when linear response itself is violated, see
\cite{baladi-2007,baladi-smania-2008} and references therein.  
Since an open system is typically modeled by a hole in the system
(that is by a region in which the dynamics stops once the trajectory enters
it), the presence of discontinuities is inevitable.

Accordingly, one could expect that the quasi-invariant measure (which describes
the long time distribution of the trajectories conditioned to the event that
they have not left the system, i.e. they have not entered the hole) and the
escape rate (that measures the rate at which trajectories leave the system)
depend in a very erratic (non-differentiable) way on the size and position
of the hole. Yet, almost nothing is known about such situations.

In the present paper we prove a general theorem providing a first order
expansion of escape rates and exchange rates in terms of the strength of the
rare effect that is investigated (whereby refining the results in
\cite{keller-liverani-1999} even-though limited to the present setting). We
derive from this theorem explicit formulae for escape rates (both in one
dimensional and two-dimensional cases) and for the exchange rate between two
quasi-invariant sets (metastability). To our knowledge such formulae were
known only for rather special cases and are completely knew in the 
generality presented here.

Just to give an impression of the wide applicability of our main result, we
list a few examples that are detailed in section~\ref{sec:applications-1D}.

Let $z$ be a periodic point of period $p$ for the doubling map $x\mapsto
2x\mod1$ and consider intervals $I_\ve\ni z$ of length $\ve$. Denote the decay
rate for the hole $I_\ve$ by $\lambda_\ve$; i.e. the Lebesgue measure of the
set of points that are not trapped by $I_\ve$ during the first $n$ iterations
of the map decreases asymptotically like $\lambda_\ve^n$. Then $\lambda_\ve$
has the following first order expansion at $\ve=0$
\begin{equation}\label{eq:doubling}
  \lambda_\ve=1-\ve\cdot\left(1-2^{-p}\right)+o(\ve)\ .
\end{equation}
Similar formulae can be obtained when the hole is a union of several
intervals, and analogous results hold of course for a coin tossing process
that is stopped once a pattern of heads and tails from a given finite
collection is observed.

We turn to continued fraction expansions. Let $m_{k,n}$
be the Lebesgue measure of those points $x\in[0,1]$ whose continued fraction
expansion up to the $n$-th digit does not contain a block of $k$ consecutive
$1$'s. For each $k$, these numbers decrease asymptotically like some
$\lambda_k^n$ and, setting $z=\frac{\sqrt5-1}2$, we show
\begin{equation}\label{eq:gauss}
      \lim_{k\to\infty}\frac{1-\lambda_{k}}{z^{2k}}
    =
    \frac{z^3(1+z^2)^2}{\ln2}\approx0.6504\,.
\end{equation}

  The above are one dimensional examples. Along the same lines one can treat
  piecewise expanding maps in higher dimensions provided the invariant density
  is not too irregular in a neighborhood of the holes. This is not guaranteed
  working only with the usual multivariate $\BV$-spaces, but rather by variants
  as developed by Blank \cite{blank-1987} or Saussol \cite{saussol}. In
  principle the present theory also applies to Anosov diffeomorphisms if one
  can devise the proper functional space setting. Unfortunately, despite
  recent progress \cite{DL, BG}, the available settings are still not adequate
  for the applications of the present results. Nevertheless, it is conceivable
  that in the near future this result could be applied e.g. to billiards.

A related but different question occurs if a system has two ergodic mixing
components that share a common part of their boundary in phase space. In that
case small perturbations can cause rare ``jumps'' over the boundary giving
thus rise to quasistationary (also called metastable or nearly invariant) behavior. As
a result the double eigenvalue $1$ corresponding to the two original mixing
components splits into a single eigenvalue $1$ and another real eigenvalue
$\lambda_\ve$ close to $1$ which characterizes the rate of exchange between
the two components, see e.g. \cite{dellnitz-junge-1999,{MSchF-2005}}. In
subsection~\ref{subsec:exchange} we consider piecewise expanding 1D-maps $T$
on the interval $[0,1]$ with two mixing components $I_1,I_2$ having a fixed
point $z$ in common. The Markov process obtained by adding to the dynamics, at
each time $n$, independent identically distributed random noise $\ve Z_n$
shows quasistationary behavior with
\begin{equation}
  \label{eq:final-exchange-intro}
  \lim_{\ve\to 0}\frac{1-\lambda_\ve}{\ve}
  =
  \frac{\beta+\alpha}2\left(1-\frac1{T'(z)}\right)\bE[|W|]
  +\frac{\beta-\alpha}2 \bE[Z_1]
\end{equation}
where $\alpha=(2m(I_1))^{-1}$, $\beta=(2m(I_2))^{-1}$
and $W:=\sum_{n=0}^\infty [T'(z)]^{-n} Z_{n+1}$.

The paper is organized as follows: in the next section we describe the general
setting and state our main Theorem~\ref{thm:main} whose proof is postponed to
section~\ref{sec:proof-main}. In section~\ref{sec:applications-1D} we apply
the theorem to decay and exchange rates of piecewise expanding interval maps
and illustrate its applicability with some specific examples: the doubling
map, the Gauss map and the generalized cusp map. The decay rate of a
two-dimensional example is studied in section~\ref{sec:coupled}. More
precisely, we study the rate at which trajectories of two coupled 1D-maps
synchronise up to some difference $\ve$ in the limit $\ve\to0$. Finally, in
section~\ref{sec:comments}, we indicate relations among between formula and
approaches to metastability in molecular dynamics \cite{MDHSch-2006}, in
oceanic structures \cite{froyland-et-al}, and with the Shannon capacity of
constrained binary codes \cite{kashyap-2003}.

\section{An abstract perturbation result}
\label{sec:abstract}
Let $(V,\|.\|)$ be a real or complex normed vector space with dual $(V',\|.\|)$ .
Consider a family $P_\ve:V\to V$ $(\ve\in E)$ of uniformly bounded linear operators where
$E\subseteq\bR$ is a closed set of parameters with $\ve=0$ as an accumulation
point. We make the following assumptions on the operators $P_\ve$:
there are $\lambda_\ve\in\bC$, $\vf_\ve\in V$, $\nu_\ve\in V'$ and linear
operators $Q_\ve:V\to V$ such that
\begin{gather}
  \tag{A1}\lambda_\ve^{-1}P_\ve=\vf_\ve\otimes\nu_\ve+Q_\ve,\label{gather:P}\\
  \tag{A2}P_\ve(\vf_\ve)=\lambda_\ve\vf_\ve,\;
  \nu_\ve P_\ve=\lambda_\ve\nu_\ve,\;Q_\ve(\vf_\ve)=0,\;\nu_\ve Q_\ve=0,\;
  \label{gather:lambda}\\
  \tag{A3}  
  \sum_{n=0}^\infty\sup_{\ve\in E}\|Q_\ve^n\|=:C_1<\infty,\label{gather:Q}
\end{gather}
(The summability condition in \eqref{gather:Q} can only be
satisfied if the operators $P_\ve$ have a uniform spectral gap. See
Remark~\ref{rem:summability} for a weakening of this requirement.)
Observe that assumptions \eqref{gather:P} and
\eqref{gather:lambda} imply $\nu_\ve(\vf_\ve)=1$ for all $\ve$.
As our ultimate goal is to prove a perturbation result for small $\ve$, it is
natural to relate the ``size'' of $\vf_\ve$ to that of $\vf_0$ by a further assumption:
\begin{gather}
  \tag{A4}\nu_0(\vf_\ve)=1\quad\text{and}\quad
  \sup_{\ve\in E}\|\vf_\ve\|=:C_2<\infty.\label{gather:phi}
\end{gather}
{}Finally we denote
\begin{equation}
  \Delta_\ve:=\nu_0((P_0-P_\ve)(\vf_0))
\end{equation}
and we make the following assumptions to control the size of the perturbation:
there is $C_3>0$ such that
\begin{gather}
  \tag{A5}\eta_\ve:=\|\nu_0(P_0-P_\ve)\|\to0\text{ as }\ve\to0,\label{gather:small1}\\
  \tag{A6}\eta_\ve\cdot\|(P_0-P_\ve)(\vf_0)\|\leq C_3|\Delta_\ve|.\label{gather:small2}
\end{gather}
Here $\eta_\ve$ denotes the norm of the linear functional
$\nu_0(P_0-P_\ve):V\to\bR$.  

The basic identity is
\begin{equation}
  \label{eq:basic1}
  \lambda_0-\lambda_\ve
  =
  \lambda_0\nu_0(\vf_\ve)-\nu_0(\lambda_\ve(\vf_\ve))
  =
  \nu_0((P_0-P_\ve)(\vf_\ve)).
\end{equation}
In view of assumptions \eqref{gather:phi} and \eqref{gather:small1} this
implies
\begin{equation}
  \label{eq:Lipschitz-at-0}
  \left|\lambda_0-\lambda_\ve\right|
  \leq
  C_2\eta_\ve,
\end{equation}
in particular, $\lim_{\ve\to0}\lambda_\ve=\lambda_0$.
The main result of this section is the following more accurate approximation for
$\lambda_0-\lambda_\ve$.
\begin{thm} \label{thm:main}
  Assume \eqref{gather:P}--\eqref{gather:small2}.
  \begin{enumerate}[a)]
  \item There is $\ve_0>0$ such that $\lambda_\ve=\lambda_0$ if $\ve\leq\ve_0$
    and $\Delta_\ve=0$.
  \item If $\Delta_\ve\neq0$ for all sufficiently small $\ve\in E$ and if
  \begin{equation}\label{eq:main-assumption}\tag{A7}
    q_k:=\lim_{\ve\to0}q_{k,\ve}:=\lim_{\ve\to0}\frac{\nu_0\left((P_0-P_\ve)P_\ve^{k}(P_0-P_\ve)(\vf_0)\right)}{\Delta_\ve}
  \end{equation}
  exists for each integer $k\geq0$, then 
  \begin{equation}
    \label{eq:main}
    \lim_{\ve\to0}\frac{\lambda_0-\lambda_\ve}{\Delta_\ve}
    =
    1-\sum_{k=0}^\infty\lambda_0^{-(k+1)} q_{k}\,.
  \end{equation}
  \end{enumerate}
\end{thm}
\begin{rem}
  \label{rem:summability}
  In section~\ref{sec:proof-main} we prove this theorem under slightly weaker
  hypothesis that may be applicable also in non-uniformly hyperbolic
  situations. Namely, we relax the summability condition 
  $\sum_{n=0}^\infty\sup_{\ve\in E}\|Q_\ve^n\|<\infty$ from \eqref{gather:Q}
  in the following way: there is a second norm $\|.\|_*\geq\|.\|$ on $V$ such
  that
  \begin{equation}
    \label{gather:Q-weak}
    \tag{A3$^*$}
    \sum_{n=0}^\infty\sup_{\ve\in E}\|Q_\ve^n\|^*=:C_1<\infty 
  \end{equation}
  where $\|Q_\ve^n\|^*:=\sup\{\|Q_\ve^n\psi\|:\|\psi\|_*\leq1\}$. We have to
  compensate this by slightly stronger assumptions on $\vf_0$, namely
  $\|\vf_0\|_*\leq C_2<\infty$ and
  \begin{equation}
    \tag{A6$^*$}
    \eta_\ve\cdot\|(P_0-P_\ve)(\vf_0)\|_*\leq C_3|\Delta_\ve|. 
    \label{gather:small2-strong}
  \end{equation}
\end{rem}

\section{Applications to piecewise expanding interval maps}
\label{sec:applications-1D}
Assume that $T:[0,1]\to[0,1]$ is piecewise monotone
with (possibly countably many) continuously differentiable branches. (This
means that each branch is continuously differentiable in the interior of its
domain so that the derivative even of a single branch may be unbounded.) Define
$g:\bR\to\bR$ by $g(x)=1/|T'(x)|$ if $x$ is in the interior of one of the
monotonicity intervals of $T$ and $g(x)=0$ otherwise, and assume that
$\|g\|_\infty<1$ and that $g$ is
of bounded variation. 
Let $\BV$ be the space of real-valued functions of bounded
variation on $[0,1]$. 
Rychlik \cite{rychlik-1983} showed that the
Perron-Frobenius operator $P$ of $T$ acting on Lebesgue equivalence classes of
functions from $\BV$ is quasi-compact.
(As $\BV$-functions have at most countably many discontinuities, two
$\BV$-functions in the same Lebesgue
equivalence class have the same discontinuities and differ at most by their
values at these countably many points, and we will not distinguish henceforth
between $\BV$-functions and their Lebesgue equivalence classes.)

If $T$ is mixing this implies that $P_0=P$ satisfies
\eqref{gather:P}--\eqref{gather:Q} for $\ve=0$ with $\nu_0=m(=\text{Lebesgue
  measure})$, $\lambda_0=1$ and $0\leq\vf_0\in\BV$.

The essential observation behind this is that for $\ve=0$ a Lasota-Yorke type
inequality \cite{lasota-yorke-1973} is satisfied: there are constants
$r\in(0,1)$ and $R>0$ such that for 
$\ve=0$, all $n\in\bN$ and all $f\in\BV$,  
\begin{equation}
  \label{eq:LY}
  \|P_\ve^n f\|\leq R\,(r^n\|f\|+\int |f|\,dm)
\end{equation}
where $\|f\|$ is the variation of the extension of $f$ to the whole real line
by setting $f(x)=0$ if $x\not\in[0,1]$.  For the applications in this section
we will assume that this inequality holds not only for $\ve=0$ but, with
uniform constants $r$ and $R$, for all $\ve\in E$. This is mostly the case
when $P_\ve$ is a small dynamical perturbation of $P_0$ -- however there are
exceptions, see \cite{keller-liverani-1999} for a more precise discussion and
references.

\subsection{Decay rates}
\label{subsec:decay}
We suppose that $T$ is mixing. Let $(V,\|.\|)$ be the space $\BV$, let
$E=[0,\ve_1]$, and consider a family $(I_\ve)_{\ve\in E}$ of compact
subintervals of $[0,1]$ such that $I_\ve\subseteq I_{\ve'}$ if
$\ve\leq\ve'$. Define the operators $P_\ve$ by $P_\ve(f)=P(f1_{[0,1]\setminus
  I_\ve})$. If $m(I_{\ve_1})$ is sufficiently small, the perturbation results
from \cite{keller-liverani-1999} apply provided \eqref{eq:LY} holds, see
\cite[section 7]{liverani-maume-2003}. In particular,
\eqref{gather:P}--\eqref{gather:phi} are satisfied for $\ve\in E$. We have
$\Delta_\ve=\nu_{0}(P(1_{I_\ve\setminus I_{0}}\vf_{0}))=\mu_0(I_\ve\setminus
I_{0})$ where $\mu_0$ is the probability measure with density $\vf_0$
w.r.t. $\nu_0$. ($\mu_0$ is indeed the equilibrium state for $\log g$ on the
``non-trapped'' set $X_{nt}:=\{x\in[0,1]:T^nx\not\in
\overset{\circ}{I}_0\,\forall n\geq0\}$. $\vf_\ve$ is also the conditionally
invariant density for the ``hole'' $I_\ve$, see
e.g. \cite{liverani-maume-2003}.)

We need to check assumptions \eqref{gather:small1} and \eqref{gather:small2}.
First note that
\begin{equation}
  \eta_\ve
  =
  \sup_{\|\psi\|\leq1}|\nu_0(P_0(\psi 1_{I_\ve\setminus I_{0}}))|
  =
  |\lambda_0|\sup_{\|\psi\|\leq1}\left|\int_{I_\ve\setminus
      I_{0}}\psi\,d\nu_0\right|
  \leq 
  |\lambda_0|\,\nu_0(I_\ve\setminus I_{0})\, .
\end{equation}
In particular, $|\nu_0(P_0-P_\ve)(\vf_0)|=|\lambda_0|\int_{I_\ve\setminus
  I_{0}}\vf_0\,d\nu_0$.  As
$\|(P_0-P_\ve)(\vf_0)\|\leq\cO(\|\vf_01_{I_\ve\setminus I_{0}}\|)$, assumptions
\eqref{gather:small1} and \eqref{gather:small2} will be satisfied if
$\nu_0(I_\ve\setminus I_{0})\to0$ when $\ve\to0$ and if
\begin{equation}\label{eq:alternative-small}
  \|\vf_01_{I_\ve\setminus I_{0}}\|
  \leq
  \text{const}\,\frac1{\nu_0(I_\ve\setminus I_{0})}\int_{I_\ve\setminus I_{0}}\vf_0\,d\nu_0\,.
\end{equation}
This condition (as well as conditions \eqref{gather:P}--\eqref{gather:phi}
discussed above) can be checked easily in most cases of interest. It is always
satisfied if $\inf\vf_0|_{I_{\ve_1}}>0$. 

\subsubsection{Holes $I_\ve$ shrinking to a point}

We specialism to the case where $I_0=\{z\}$ for some $z\in[0,1]$ so that $P_0$
is indeed the Perron-Frobenius operator $P$ for $T$ and $\lambda_0=1$, and we
assume for simplicity that $T$ and also the invariant density $\vf_0$ are
continuous at $z$. We consider $I_\ve$ with length $\ve$, so $m(I_\ve\setminus
I_0)=\ve$, and we assume that $\Delta_\ve>0$.  Here are a few examples:
\begin{description}
\item[The doubling map] $T(x)=2x\text{ mod\,}1$ with $\vf_0(x)=1$.
\item[The Gauss map] $T(x)=\frac1x\text{ mod\,}1$
 with $\vf_0(x)=\frac1{\ln 2}\frac1{1+x}$
\item[The generalized cusp map] $T_\gamma(x)=1-|2x-1|^\gamma$ for some
  $\gamma\in(\frac12,1]$. As
  $|T_\gamma'(x)|=\frac{2\gamma}{|2x-1|^{1-\gamma}}\geq2\gamma$, this map is a
  uniformly expanding map with two full branches. The weight function
  $g(x)=|T_\gamma'(x)|^{-1}$ has two monotone bounded branches and is clearly
  of bounded variation. (Observe that $T_1$ is just the tent map. $T_{1/2}$ is
  known as the cusp map.  It has $x=0$ as a neutral fixed point and is not
  covered by the present setting.)  The invariant density $\vf_0(x)$ of
  $T_\gamma$ behaves like $\frac{\vf_0(1/2)}\gamma (1-x)^{\frac1\gamma-1}$
  near $x=1$, so it has a zero at $x=1$ if $\gamma<1$.\footnote{Here is a
    sketch of the argument: As the map $T_\gamma$ has full branches, the
    invariant density $\vf_0=\lim_{n\to\infty}P_0^n1$ is continuous. Also
    $\vf_0(\frac12)>0$, because otherwise $\vf_0(x)=0$ for all
    $x\in\bigcup_nT_\gamma^{-n}\{\frac12\}$, and this set is dense in
    $[0,1]$. Therefore, for $x$ close to $1$,
    \begin{displaymath}
      \vf_0(x)=P_0\vf_0(x)
      \sim\frac{\vf_0(1/2)}\gamma(1-x)^{\frac1\gamma-1}
    \end{displaymath}
}
\end{description}
In all three examples, $\Delta_\ve=\mu_0(I_\ve)>0$. In the first two examples,
condition \eqref{eq:alternative-small} is clearly satisfied because
$\inf\vf_0|_{I_\ve}>0$ if $\ve$ is sufficiently small. For the generalized cusp
map the same is true if $z\neq1$. In case $z=1$,
$\|\vf_01_{I_\ve}\|=2\,\text{const}_\gamma\,\ve^{\frac1\gamma-1}$ and
$\int_{I_\ve}\vf_0\,dm=\text{const}_\gamma\int_{1-\ve}^1(1-x)^{\frac1\gamma-1}\,dx=\text{const}_\gamma\,\gamma\,\ve^{\frac1\gamma}$
so that \eqref{eq:alternative-small} is satisfied as well.

Let
\begin{displaymath}
  U_{k,\ve}:=T^{-1}([0,1]\setminus I_\ve)\cap\dots\cap
  T^{-k}([0,1]\setminus I_\ve)\cap T^{-(k+1)}I_\ve.
\end{displaymath}
As $q_{k,\ve}=\mu_0(I_\ve\cap U_{k,\ve})/\mu_0(I_\ve)$, we find:
\begin{description}
\item[If $z$ is not periodic] then $U_{k,\ve}=\emptyset$ for sufficiently
  small $\ve$ so that $q_k=\lim_{\ve\to0}q_{k,\ve}=0$ for all $k$. Therefore,
  \begin{equation}\label{eq:limit-not-periodic}
    \lim_{\ve\to0}\frac{1-\lambda_\ve}{\mu_0(I_\ve)}=1,\text{ in
      particular }
    \lim_{\ve\to0}\frac{1-\lambda_\ve}{m(I_\ve)}=\vf_0(z)\,.
  \end{equation}
\item[If $z$ is periodic with period $p$] then $U_{k,\ve}=\emptyset$ for
  sufficiently small $\ve$ except if $k=p-1$ so that
  \begin{gather}
    \lim_{\ve\to0}\frac{1-\lambda_\ve}{\mu_0(I_\ve)}=1-\lim_{\ve\to0}\frac{\mu_0(I_\ve\cap
        T^{-p}I_\ve)}{\mu_0(I_\ve)}
    =1-\frac1{|(T^p)'(z)|}\,,\\
    \text{in particular }
    \lim_{\ve\to0}\frac{1-\lambda_\ve}{m(I_\ve)}
    =
    \vf_0(z)\left(1-\frac1{|(T^p)'(z)|}\right)\,.
    \label{eq:limit-periodic}
  \end{gather}
\end{description}
Formulas \eqref{eq:limit-not-periodic} and
\eqref{eq:limit-periodic} imply that the function $\ve\mapsto\lambda_\ve$ is
differentiable at $\ve=0$. Note, however, that in general it is
non-differentiable at other values of $\ve$, see section
\ref{subsubsec:big-hole} below.

We look more explicitly at the above three examples. Recall that $I_0=\{z\}$.
\begin{description}
\item[The doubling map]
  $\lim_{\ve\to0}\frac{1-\lambda_\ve}{m(I_\ve)}=1-2^{-p}$ if
  $T^p(z)=z$. This is \eqref{eq:doubling}.
\item[The Gauss map] a) Consider $z=0$ and $\epsilon\in
  E:=\{\frac12,\frac13,\frac14,\dots\}\cup\{0\}$. Let $I_\ve:=[0,\ve]$. In
  terms of the continued fraction algorithm this means that the expansion
  stops as soon as a digit $\geq\ve^{-1}$ is generated. In this case it is
  easy to see that
  \begin{displaymath}
    \begin{split}
      \mu_0(I_\ve\cap U_{k,\ve})
      &\leq
      \frac1{\ln2}m(I_\ve\cap T^{-(k+1)}I_\ve)
      =
      \frac1{\ln2}\int_{I_\ve}P_0^k(P_01_{I_\ve})\,dm\\
      &=
      \cO(\ve\,\mu_0(I_\ve))      
    \end{split}
  \end{displaymath}
  so that $q_{k,\ve}=\cO(\ve)$ and hence $q_k=0$ for all $k$. Hence
  $\lim_{\ve\to0}\frac{1-\lambda_\ve}{\ve}=\frac1{\ln2}$.\\
  b) For the same map we consider $z=\frac{\sqrt5-1}2$ which is the rightmost
  fixed point of $T$. We have $T'(z)=-z^{-2}$. As $-z$ and $z^{-1}$ are the
  two zeros of $x^2-x-1$, it is obvious that $1-\frac1{|T'(z)|}=z$. Hence, for
  intervals $I_\ve$ of length $\ve$ around $z$ we have
  $\lim_{\ve\to0}\frac{1-\lambda_\ve}\ve=\frac1{\ln2}\frac{z}{1+z}=\frac{z^2}{\ln2}$.\\
  c) Denote $f(x)=\frac1x-1$. Then $f$ is the rightmost branch of $T$, and the
  interval around $z$ which is mapped by $T^k$ onto $(0,1)$ has endpoints
  $f^{-k}(1)$ and $f^{-k}(0)$. Denote the length of this interval by $\ve_k$
  and the interval itself by $I_{\ve_k}$. As $-z$ and $z^{-1}$ are the
  eigenvalues of the the coefficient matrix of $f^{-1}$, a calculation shows
  that $\ve_k=z^{2k+1}(1+z^2)^2(1+\cO(z^{2k+2}))$.\footnote{Denote the coefficient matrix ${0\;1\choose 1\;1}$ of
    $f^{-1}$ by $M$ and let $M^k=:{a_k\;b_k\choose c_k\;d_k}$. Then
    $f^{-k}(x_1)-f^{-k}(x_0)=\det(M^k)\frac{x_1-x_0}{(c_kx_1+d_k)(c_kx_0+d_k)}$
    so that $\ve_k=|f^{-k}(1)-f^{-k}(0)|=\frac1{|(c_k+d_k)d_k|}$. As
    $c_{k+1}=d_k$ and $d_{k+1}=(c_k+d_k)$, we have
    $\ve_k=\frac1{|d_{k+1}d_k|}$ and $d_{k+1}=d_k+d_{k-1}$. With $d_0=d_1=1$
    this yields $d_k=\frac1{1+z^{-2}}(z^{-(k+2)}+(-z)^k)$. Hence
    \begin{displaymath}
      \ve_k
      =
      \frac{z^{2k}(1+z^{-2})^2}{(z^{-3}+(-1)^{k+1}z^{2k+1})(z^{-2}+(-1)^kz^{2k})}
      =
      z^{2k+1}(1+z^2)^2(1+\cO(z^{2k+2}))
    \end{displaymath}
}
In terms of the continued fraction algorithm the hole $I_{\ve_k}$ means that an expansion stops at time $n+k$ as soon as at least $k$ consecutive digits $1$ are generated.  So it is natural to rewrite the limit from b) as in formula \eqref{eq:gauss}, namely
  \begin{equation}
    \label{eq:gauss-fix}
    \lim_{k\to\infty}\frac{1-\lambda_{\ve_k}}{z^{2k}}
    =
    \frac{z^3(1+z^2)^2}{\ln2}\approx0.6504
  \end{equation}

\item[The generalized cusp map] We focus on $z=1$, where the invariant
  density $\vf_0$ vanishes, and consider holes $I_\ve=[1-\ve,1]$. Then, for
  $\ve$ close to $0$, we have
  $\mu_0(I_\ve)\sim\frac{\vf_0(1/2)}\gamma\int_{1-\ve}^1(1-x)^{\frac1\gamma-1}\,dx=\vf_0(\frac12)\,\ve^{\frac1\gamma}$
  so that
  $\lim_{\ve\to0}\frac{1-\lambda_\ve}{\ve^{1/\gamma}}=\vf_0(\frac12)$. 
\end{description}

\subsubsection{Holes shrinking to a nontrivial hole}
\label{subsubsec:big-hole}

We assume now that $I_0$ is a nontrivial interval of some fixed length
$\ell>0$ and the intervals $I_\ve\supseteq I_0$ have length $\ell+\ve$. To
simplify the discussion we assume more specifically that
$I_\ve=[a-\ve,a+\ell]$ where $a\in[0,1]$ is a continuity point of $T$ and also
of $\vf_0$. Assume furthermore that $a$ is not periodic for $T$ (the periodic
case can be dealt with analogously). Now, as $\mu_0$ is supported by the
non-trapped set $X_{nt}$, we have in particular $\mu_0(I_0)=0$ and hence
$\Delta_\ve=\mu_0(I_\ve)$. It follows from Theorem~\ref{thm:main} that either
$\mu_0(I_\ve)=0$ and hence $\lambda_\ve=\lambda_0$ for all sufficiently small
$\ve$, or $\lim_{\ve\to0}\frac{\lambda_0-\lambda_\ve}{\mu_0(I_\ve)}=1$. But
observe that $\mu_0$ is of fractal nature, so typically $\mu_0(I_\ve)$ depends
on $\ve$ in a devil's staircase manner. Hence either $\lambda_\ve=\lambda_0$
for small $\ve$ (which happens if $a$ itself is trapped), or
$\lim_{\ve\to0}\frac{\ve}{\mu_0(I_\ve)}=0$ and $\ve\mapsto\lambda_\ve$ is not
differentiable at $\ve=0$.


\subsection{Exchange rates}
\label{subsec:exchange}
We suppose that $T$ has two ergodic components and that its restriction to  each of these components is mixing. 
So the eigenvalue $1$ of $P$ has (geometric) multiplicity $2$ and the rest of its spectrum is contained in a
disk of radius smaller than some $\gamma\in(0,1)$.  Let $(\Pi_\ve)_{\ve\in E}$
be a family of Markov operators close to the identity with $\Pi_0=\Id$, and
denote $P_\ve:=P\circ\Pi_\ve$. (One could as well consider $\Pi_\ve\circ P$ since
that operator has the same eigenvalues as $P_\ve$.) Under rather weak regularity assumptions on the
$\Pi_\ve$, the spectral perturbation results from \cite{keller-liverani-1999}
apply again. This is true, for example, if the $\Pi_\ve$ are convolutions with
smooth densities $k_\ve(x)=\ve^{-1}k(\ve^{-1}x)$ (modeling random
perturbations) or if they are conditional expectations w.r.t. $m$ and a finite
partition into intervals of length $\ve$ (modeling Ulam's discretization
scheme.) As the $P_\ve$ are also Markov operators, this means that $1$ is
an isolated eigenvalue of each $P_\ve$. If it has multiplicity $2$ there is
nothing more to say about it. If it is a simple eigenvalue, however, then
there is a second simple eigenvalue $\lambda_\ve$ close to $1$ to which we
will apply Theorem~\ref{thm:main}.

Let $(V,\|.\|)$ be the
space
\begin{displaymath}
  \BV_0:=\{f\in\BV: m(f)=0\}.
\end{displaymath}
$\BV_0$ is invariant under all $P_\ve$, and the previous discussion implies
that assumptions \eqref{gather:P}--\eqref{gather:phi} are satisfied. More
precisely, $\lambda_0=1$, and there is an increasing function $\chi:[0,1]\to\{-1,1\}$ such that
$\chi\circ T=\chi$, $|\vf_0|=\chi\vf_0$ is an invariant density for $P=P_0$,
and $\nu_0=\chi m$, so that $\mu_0:=\vf_0\nu_0=|\vf_0|m$ is an invariant
probability measure that gives equal mass to both ergodic components of $T$.
Let $\tilde I_1=\{\chi=-1\}$ and $\tilde I_2=\{\chi=1\}$ be the two invariant components of $T$.

In order to apply Theorem \ref{thm:main} let
\begin{displaymath}
  p_\ve(x):=\frac12\left(1-\left(\chi(x)\cdot(\Pi_\ve^*\chi)(x)\right)\right)
\end{displaymath}
where $\Pi_\ve^*$ is the dual of $\Pi_\ve$ with respect to Lebesgue measure on
$[0,1]$: It is easy to see that $p_\ve(x)$ is the probability, that the
Markovian dynamics $\Pi_\ve$ move the system from the state $x$ (that belongs
to one of the two invariant components of $T$) to some state in the other
component.

We have to check assumptions \eqref{gather:small1} and
\eqref{gather:small2}. Routine calculations show that 
$\nu_0(P_0-P_\ve)(\psi)=2m(\chi\cdot p_\ve\cdot \psi)$ 
for each $\psi\in
V$. This implies
\begin{equation}
  \label{eq:Delta-exchange}
  \Delta_\ve=\nu_0(P_0-P_\ve)(\vf_0)=2m(|\vf_0|\cdot p_\ve)=2\mu_0(p_\ve)
\end{equation}
and 
$\nu_0(P_0-P_\ve)(\psi)=2m(\chi p_\ve\psi)\leq 2\|\psi\|\int_0^1|p_\ve|dm$ 
so that $\eta_\ve\leq2m(p_\ve)$. Therefore we require that the average
probability $m(p_\ve)$ to change the invariant component under the action of
$\Pi_\ve$ tends to $0$ as $\ve\to0$ and that
\begin{equation}\label{eq:condA6}
  \|(\Id-\Pi_\ve)(\vf_0)\|
  \leq
  \text{const}\cdot\frac1{m(p_\ve)}\int|\vf_0|p_\ve\,dm\ .
\end{equation}
In the following we will assume
\[
\inf|\vf_0|>0.
\]
This trivially implies \eqref{eq:condA6} although the latter can be verified in many other cases. 
It remains to check assumption \eqref{eq:main-assumption}. Observing
\eqref{eq:Delta-exchange} and
\begin{equation}
  \begin{split}
    \nu_0\left((P_0-P_\ve)P_\ve^{k}(P_0-P_\ve)(\vf_0)\right)
    &=
    m((\chi-\Pi_\ve^*\chi)\cdot P_\ve^k(P_0-P_\ve)(\vf_0))\\
    &=
    2m\left(p_\ve\,\chi\cdot P_\ve^kP_0(\vf_0-\Pi_\ve\vf_0)\right)
  \end{split}
\end{equation}
we get the following expression for the $q_{k,\ve}$:
\begin{equation}
  q_{k,\ve}=
  \frac1{m(|\vf_0|\,p_\ve)}\,m\left(p_\ve\,\chi\cdot P_\ve^kP_0(\vf_0-\Pi_\ve\vf_0)\right)
\end{equation}

The evaluation of the limit as $\ve\to0$ depends strongly on the
details of the map $T$ and of the perturbation. 
We therefore make some further simplifying assumptions:
\begin{itemize}
\item The $\Pi_\ve$ are local perturbations, i.e., for each $x$,
  $\Pi_\ve\delta_x$ is supported in a $C\ve$-neighborhood of $x$.
\item $T$ is continuous.
\end{itemize}
As the restrictions of $T$ to its two ergodic components are mixing, the
continuity of $T$ implies that the non-wandering part of these components are
just two single intervals $I_i\subset \tilde I_i$.  If these intervals do not have a common end point, then
$\Delta_\ve=2\mu_0(p_\ve)=0$ for small $\ve$ so that $\lambda_\ve=1$ for such
$\ve$ by our main theorem.  Otherwise $I_1$ and $I_2$ have a common endpoint
$z$.  Since two interval can have at most one common endpoint and since the
map is continuous, it follows by the invariance of $I_1,I_2$ that $z$ is a
fixed point. In this case, $p_\ve(x)=0$ unless $x$ belongs to the
$C\ve$-neighborhood of $z$. As an example let us consider the special (but
still rather general) class of examples characterized by the following
properties
\begin{itemize}
\item $\overline{I_1\cup I_2}=[0,1]$.\footnote{If the wandering part is
    present, the final result still holds with $I_i$ substituted by
    $\cup_{n\in\bN}T^{n}I_i$ in \eqref{eq:alpha-beta}.}
\item Assume $\Pi_\ve f(x)=\int_{0}^1K_\ve (y,x)f(y)dy$ where $K_\ve$ is a
  positive kernel such that, for all $y\in[0,1]$, $\int_0^1K_\ve(y,x)dx=1$.
    In order to satisfy assumptions \eqref{gather:P}--\eqref{gather:phi} one
    should suppose that the kernels are bistochastic or close to convolution
    kernels in the sense of \cite[Corollary 3.20]{blank-keller-1997}.
  \item There exists $a>0$ such that $K_\ve(y,x)=\ve^{-1} K(\ve^{-1}(x-y))$
    provided $|z-y|\leq a$.  Here $K$ is a smooth probability density
    supported in $[-1,1]$.\footnote{One can consider the more general case
      $K_\ve(y,x)=\ve^{-1} \tilde K(y,\ve^{-1}(x-y))$, for some smooth
      function $\tilde K$. The final formula then holds with $K(\cdot)$
      replaced by $\tilde K(z,\cdot)$.}
\item $\vf_0$ is continuous in each ergodic component and $T$ is differentiable at $z$.
\end{itemize}
Note that, since $|\vf_0|$ must give the same weight to the two ergodic
components it will, in general, be discontinuous at $z$. Let $\alpha,\beta$ be
the left and right limit respectively. Then, introducing coordinates
$x=z+\ve\zeta$ and setting $\theta(y)=\sign y$, we have
$\vf_0(z+\ve\zeta)=\frac{\beta+\alpha}2\theta(\zeta)+\frac{\beta-\alpha}2+o(1)$,
uniformly for $\zeta$ in a compact set. Next, for $\ve$ small enough,
\[
p_\ve(z+\ve\zeta)=\frac12\int_\bR K(y-\zeta)[1-\theta(y)\theta(\zeta)]dy=:p(\zeta).
\]
Note that $p\geq 0$ and $p(\zeta)=0$ if $|\zeta|> 1$. Accordingly,
\[
\begin{split}
\mu_0(p_\ve)
&=m(|\vf_0|p_\ve)=
\ve \frac{\beta-\alpha}2m(\theta p)+\ve\frac{\beta+\alpha}2 m(p)+o(\ve)=:\ve\Gamma+o(\ve).
\end{split}
\]
In addition, for each function $\Psi_\ve$ such that $\Psi_\ve(z+\ve\zeta)=\psi(\zeta)+o(1)$, for some fixed compact support function $\psi$, holds, in the limit $\ve\to0$, $(\Pi_\ve\Psi_\ve)(z+\ve\zeta)=\int_\bR
K(\zeta'-\zeta)\psi(\zeta')d\zeta'+o(1)$ and $(P_0\Psi_\ve)(z+\ve\zeta)=\Lambda\cdot\psi(\Lambda\zeta)+o(1)$ where $\Lambda:=\frac1{T'(z)}>0$.  Hence,
\[
\begin{split}
(P_\ve \Psi_\ve)(z+\ve\zeta)&=\Lambda\int_\bR K(\zeta\Lambda-\zeta')\psi(\zeta')d\zeta'+o(1)\\
&=:(\cK \psi)(\zeta)+o(1).
\end{split}
\]
The above setting applies to $\Psi_\ve=P_0(\Id-\Pi_\ve)\vf_0$, namely
\[ [P_0(\Id-\Pi_\ve)\vf_0](z+\ve\zeta)= \frac{\beta+\alpha}2
[\Lambda\theta(\zeta)-\cK\theta(\zeta)]+o(1).
\]
Indeed, $\cK\theta(\zeta)=\Lambda\theta(\zeta)$ for $|\zeta|\geq\Lambda^{-1}$, hence
$\Lambda \theta-\cK\theta$ is compactly supported.
Thus
\[
q_k=\frac{ (\beta+\alpha)\langle\theta p, \cK^k(\Lambda\Id-\cK)\theta\rangle}{2\Gamma}.
\]
Since the operator $\cK$ has $L^\infty$ norm smaller than $|\Lambda|$, the latter equality implies
\[
\lim_{\ve\to
  0}\frac{\lambda_0-\lambda_\ve}{\ve}=2\Gamma-\frac{(\beta+\alpha)}2\sum_{n=0}^\infty\langle
2\theta p, \cK^n(\Lambda\Id-\cK)\theta\rangle.
\]
On the other hand, a direct computation
(observing the fact that
$\theta(y)=\theta(\Lambda y)$) shows that $2\theta p=(\Id-\cK^*)\theta$, so
\[
\lim_{\ve\to 0}\frac{\lambda_0-\lambda_\ve}{\ve}=
2\Gamma-\frac{\beta+\alpha}2\left[\langle \theta,
  (\Lambda\Id-\cK)\theta\rangle-\lim_{n\to\infty}\langle(\cK^*)^n\theta,(\Lambda\Id-\cK)\theta\rangle\right].
\]
Note that $(\cK^*)^n\theta$ converges pointwise to a function $\theta_\infty$
such that $\theta-\theta_\infty$ is supported in the interval
$[-(1-\Lambda)^{-1},(1-\Lambda)^{-1}]$, see \eqref{eq:theta-infinity}. In
particular, $\cK^*\theta_\infty=\theta_\infty$. Then,
\begin{equation}
  \label{eq:final-exchange}
  \begin{split}
    \lim_{\ve\to 0}\frac{1-\lambda_\ve}{\ve}
    &= 
    2\Gamma-\frac{\beta+\alpha}2\langle \theta-\theta_\infty,
      (\Lambda\Id-\cK)\theta\rangle\\
    &=
    2\Gamma-\frac
    {\beta+\alpha}2\left[\langle(\Id-\cK^*)(\theta-\theta_\infty),\theta\rangle
      +(\Lambda-1)\langle\theta-\theta_\infty,\theta\rangle
    \right]\\
    &=
    2\Gamma-\frac
    {\beta+\alpha}2\left[\langle 2\theta p,\theta\rangle
      +(\Lambda-1)m(1-\theta\theta_\infty)\right]\\
    &=
    \frac{\beta+\alpha}2\left(1-\frac1{T'(z)}\right)m(1-\theta\theta_\infty)
    +\frac{\beta-\alpha}2 m(2\theta p)
  \end{split}
\end{equation}

To make the formula more explicit and transparent let us make some further
remarks. The original dynamical system has a natural invariant measure defined
by $h=\lim_{n\to\infty}\frac 1n\sum_{k=0}^{n-1}P_0^n 1$.  By our assumptions
$h$ is continuous at $z$ and $\{h,\chi h\}$ form a basis for the eigenspace of
the eigenvalue one of the operator $P_0$. Thus $\vf_0=\frac{\alpha+\beta}2\chi
h+\frac{\beta-\alpha}2 h$. Remember that $\alpha,\beta$ are chosen so that
$\int_{I_1} |\vf_0|\,dm=\int_{I_2}|\vf_0|\,dm=\frac 12$. Hence $ \frac
12=\alpha\int_{I_1}h\,dm = \alpha\lim_{n\to\infty}\frac
1{n}\sum_{k=0}^{n-1}\int_0^1\frac{1-\chi}2\circ T^n\,dm = \alpha m(I_1) $ so
that
\begin{equation}\label{eq:alpha-beta}
  \alpha=(2m(I_1))^{-1}\text{ and, analogously, }\beta=(2m(I_2))^{-1}\,.
\end{equation}

To describe the meaning of the two factors involving $\theta$ and
$\theta_\infty$, let $Z$ be a random variable whose distribution has probability density $K$. Then
\begin{equation}
  m(2\theta p)
 =
 -2\,\bE[Z]\ .
\end{equation}
Next let $Z_1,Z_2,\dots$ be independent copies of $Z$. The kernel $\cK^*$
describes a Markov process
\begin{displaymath}
  X_n
  =
  \Lambda^{-1}(X_{n-1}+Z_n)
  =\dots=
  \Lambda^{-n}\left(X_0+\sum_{k=1}^n\Lambda^{k-1}Z_k\right)\,.
\end{displaymath}
The asymptotic behavior of the process $(X_n)$ is determined by the random
variable $W:=\sum_{k=1}^\infty\Lambda^{k-1}Z_k$.  Indeed, let
$X_0=\zeta$. Then $X_n\to+\infty$ if $W>-\zeta$ and $X_n\to-\infty$ if
$W<-\zeta$. (As the $Z_k$ have density, $W=0$ has probability $0$.) As
$((\cK^*)^n\theta)(\zeta)$ is the conditional expectation of $\theta(X_n)$
given $X_0=\zeta$, it follows readily that
\begin{equation}
  \label{eq:theta-infinity}
\begin{split}
  \theta_\infty(\zeta)&=\Prob(X_n\to+\infty|X_0=\zeta)-\Prob(X_n\to-\infty|X_0=\zeta)\\
  &=1-2\,\Prob(W<-\zeta)\,.
\end{split}
\end{equation}
 Hence
\begin{equation}
  m(1-\theta\theta_\infty)=2\,
  \int_0^\infty\left[\bP(W>\zeta)+\bP(-W>\zeta)\right]\,d\zeta
  =2\bE[|W|]\ .
\end{equation}
Note that 
$(1-1/T'(z))m(1-\theta\theta_\infty)=2(1-\Lambda)\bE[|W|]\geq2(1-\Lambda)|\bE[W]|=2|\bE[Z]|=|m(2\theta
p)|$, so
the r.h.s of
\eqref{eq:final-exchange} is clearly positive.

We finish this section with a comment on the term `\emph{exchange rate}. Let
$A_\ve^+:=\{\vf_\ve>0\}$, $A_\ve^-:=\{\vf_\ve<0\}$, and $\tilde
p_\ve:=1_{A_\ve^+}\cdot P_\ve^*1_{A_\ve^-}+1_{A_\ve^-}\cdot
P_\ve^*1_{A_\ve^+}$. $\tilde p_\ve(x)$ is the probability to exchange the sets
$A_\ve^\pm$ under the action of $P_\ve^*$. Now Proposition 5.7 from
\cite{dellnitz-junge-1999} can be rephrased in our setting as
$1-\lambda_\ve=2\int\tilde p_\ve|\vf_\ve|\,dm$,\footnote{The operator $P$ in
  \cite{dellnitz-junge-1999} corresponds to our $P_\ve$ and the signed measure
  $\nu$ to our $\vf_\ve m$.}  so it is nearly twice the ``stationary exchange
rate'' $\int\tilde p_\ve h_\ve\,dm$ where $h_\ve=P_\ve h_\ve$ is the unique
invariant probability density of the perturbed system.  If all $A_\ve$ are identical
(e.g. under suitable symmetry assumptions on the system as in \cite[Corollary
5.9]{dellnitz-junge-1999}), then $\tilde p_\ve$ coincides with $p_\ve$ from
above.

\section{An application to two coupled interval maps}
\label{sec:coupled}
Let $T:[0,1]\to[0,1]$ be a mixing piecewise expanding map as in
section~\ref{sec:applications-1D}. To simplify the discussion we assume that
$\gamma:=\inf|T'|>4$. Let $M:=[0,1]^2$ and define, for
$\delta\in[0,\frac14-\frac1\gamma)$,  the two-dimensional coupled map
\begin{displaymath}
  \hat T:M\to M,\quad \hat T(x,y)=((1-\delta)T(x)+\delta T(y),(1-\delta)T(y)+\delta
  T(x))\ .
\end{displaymath}
It is uniformly piecewise expanding with minimal expansion strictly larger
than $2$ in the sense that
\begin{displaymath}
  \|(D\hat T)^{-1}\|\leq\frac1{\gamma(1-2\delta)}<\frac14\ .
\end{displaymath}
As discussed in great detail in \cite{keller-liverani-2005} there is
$\delta_1\in(0,\frac12-\frac1\gamma]$ such that, for
$\delta\in[0,\delta_1]$, $\hat T$ is mixing in the sense that its
Perron-Frobenius operator $\hat P:\BV(M)\to\BV(M)$ has a unique invariant
probability density $\hat h$ and a spectral gap. Here $\BV(M)$ is the space of
functions of bounded variation on $\bR^2$ that vanish outside $M$.

For $\ve\in E:=[0,\ve_1]$ let $S_\ve:=\{(x,y)\in M:|x-y|\leq\ve\}$. If we
interpret $S_\ve$ as a hole in the phase space $M$, this means that we stop a
trajectory as soon as the two components have synchronized up to a difference
of at most $\ve$. The corresponding Perron-Frobenius operator $\hat
P_\ve:\BV(M)\to\BV(M)$ is defined by $\hat P_\ve(\psi)=\hat
P(\psi\cdot 1_{M\setminus S_\ve})$. Denote the (two-dimensional) variation of a
function $\psi\in\BV(M)$ by $\Var(\psi)$. It is easy to check that $\Var(\hat
P_\ve\psi)\leq2\Var(\hat P\psi)$ so that the family of operators $\hat P_\ve$
satisfies a uniform Lasota-Yorke inequality. (Observe that we made the
generous assumption $\gamma>4$ and consult \cite{keller-liverani-2005}.) In
view of the spectral stability results of \cite{keller-liverani-1999},
assumptions \eqref{gather:P} -- \eqref{gather:phi} are satisfied with
$\nu_0=m$ (the Lebesgue measure on $M$), $\vf_0=\hat h$, $\mu_0=\hat h m$ and
$\lambda_0=1$.

We turn to assumptions \eqref{gather:small1} and \eqref{gather:small2}.
Observe first that 
\begin{equation}
  \label{eq:2Deta}
  \nu_0(P_0-P_\ve)(\psi)=m(\psi 1_{S_\ve})\leq  C\ve\Var(\psi)\ .
\end{equation}
(The constant $C$ depends on the details of the definition of the variation.)
So in particular $\eta_\ve\leq C\ve$ and \eqref{gather:small1} is
satisfied. As $\hat h$ is of bounded variation, we may assume that it is
regularized along the diagonal of $M$ in the sense that for
$1$D-Lebesgue-almost every $x$ the value $\hat h(x,x)$ is the average of the
limits of $\hat h(x-u,x+u)$ and $\hat h(x+u,x-u)$ as $u\searrow 0$. In view of
\eqref{eq:2Deta} we therefore conclude
\begin{equation}
  \label{eq:2DDelta}
  \lim_{\ve\to0}(2\ve)^{-1}\Delta_\ve
  =
  \lim_{\ve\to0}(2\ve)^{-1}\int_{S_\ve}\hat h\,dm
  =
  \int_0^1\hat h(x,x)\,dx\ .
\end{equation}
As $\Var(\hat h1_{S_\ve})\leq2\Var(\hat h)$, we conclude that
\eqref{gather:small2} is satisfied if $\int_0^1\hat h(x,x)\,dx>0$.
It remains to evaluate the $q_k$. As in section~\ref{sec:applications-1D} let
\begin{displaymath}
  \hat U_{k,\ve}:=\hat T^{-1}(M\setminus S_\ve)\cap\dots\cap\hat
  T^{-k}(M\setminus S_\ve)\cap\hat T^{-(k+1)}S_\ve\ .
\end{displaymath}
Then $q_{k,\ve}=\mu_0(S_\ve\cap\hat U_{k,\ve})/\mu_0(S_\ve)$ and, since the
diagonal of $M$ is invariant under $\hat T$, we find
\begin{displaymath}
  q_0=\lim_{\ve\to0}q_{0,\ve}=\frac1{\int_0^1\hat h(x,x)\,dx}\int_0^1\hat h(x,x)\,\frac1{(1-2\delta)|T'(x)|}\,dx
\end{displaymath}
and $q_k=\lim_{\ve\to0}q_{k,\ve}=0$ for all $k\geq1$. So finally,
\begin{equation}
  \label{eq:2Dresult}
  \lim_{\ve\to0}\frac{1-\lambda_\ve}{2\ve}
  =
  \int_0^1\hat h(x,x)\,\left(1-\frac1{(1-2\delta)|T'(x)|}\right)\,dx
\end{equation}

\section{Related results}
\label{sec:comments}

\subsection{Metastable states in molecular dynamics and oceanic  structures}
Phase space methods to characterize biomolecular conformations as metastable
states are used in molecular dynamics (see e.g. \cite{MDHSch-2006} and
references cited there). Very roughly, if the Markov operator $P_\ve$
describes the discrete time evolution of such a system in a fixed time scale
and if $X=D_1\cup D_2$ is a decomposition (up to null sets) of the underlying
phase space, then the \emph{metastability measure} of this decomposition is
defined as
$\meta(D_1,D_2)=\frac12\left[\mu_\ve({D_1})^{-1}\int_{D_1}P_\ve1_{D_1}d\mu_\ve+\mu_\ve({D_2})^{-1}\int_{D_2}P_\ve1_{D_2}d\mu_\ve\right]$. Theorem~1
of \cite{MDHSch-2006} relates $\meta(D_1,D_2)$ to the second eigenvalue of
$P_\ve$ in a way very similar to formula~\eqref{eq:final-exchange}. Reference
\cite{froyland-padberg-2008} is an up-to-date review of the phase space
decomposition approach to metastability in general flow dynamical systems, and
\cite{froyland-et-al} is an application of these ideas to the detection of
coherent oceanic structures. In the framework of weakly coupled rapidly mixing
Markov chains, reference \cite{MSchF-2005} also relates the second largest
eigenvalue of a system to the exchange probabilities between its components.

\subsection{Shannon capacity of constrained systems of binary sequences}
In information theory, the topological entropy of subshifts of $\{0,1\}^\bN$
that are determined by a (short) list $\bL_m=(B_1^{(m)},\dots,B_p^{(m)})$ of
distinct blocks of length $m$ which are not allowed to occur \cite{lind-1989}
is called the Shannon capacity of the system. It is closely related to the
rate of periodic prefix-synchronized (PPS) codes with markers $B_i^{(m)}$,
$i=1,\dots,p$ (see e.g. \cite{kashyap-2003}). For each sequence
$\bL_1,\bL_2,\dots$ of such lists with fixed length $p$ there are a subsequence $(m_j)$
and $z_1,\dots,z_p\in\{0,1\}^\bN$ such that the $B_i^{(m_j)}$ converge to
$z_i$ as $j\to\infty$. We will assume without loss that the full sequences
$(B_i^{(m)})_m$ converge. As the full two-shift is isomorphic (for each
invariant measure of positive entropy) to the doubling map, the shift
constrained by the forbidden blocks in $\bL_m$ is isomorphic to the doubling
map $T$ with ``hole'' $I_m$ being the union of those monotonicity intervals of
$T^m$ labeled by the words in $\bL_m$. Hence the topological entropy
$h(\bL_m)$ of this shift equals $\log(2\lambda_m)$ where $\lambda_m$ is the
leading eigenvalue of the Perron-Frobenius operator of $T$ with hole $I_m$,
compare subsection~\ref{subsec:decay}.

If the limit points $z_i$ belong to $B_i^{(m)}$ for all $i$ and $m$, one can analyses the
situation just as in subsection~\ref{subsec:decay}. Some elementary reasoning
yields the following: for $i=1,\dots,p$ let
$\ell(i)=\min\{j\geq1:T^jz_i\in\{z_1,\dots,z_p\}\}$ with the convention that
$\ell(i)=+\infty$ if no such $j$ exists. Then
\begin{displaymath}
  \lim_{m\to\infty}\frac{\log
    2-h(\bL_m)}{2^{-m}}=\sum_{i=1}^p(1-2^{-\ell(i)})\ .
\end{displaymath}
This is minimal when all $\ell(i)=1$, e.g. if the $z_i$ form just one periodic
orbit. In that case $h(\bL_m)=\log2-p2^{-(m-1)}+o(2^{-m})$, which supports the
conjecture on the precise values of $h(\bL_m)$ for $p=2^k$ and $m\geq k+1$
stated in \cite{kashyap-2003}.

\section{Proof of the main theorem}
\label{sec:proof-main}
As announced in Remark~\ref{rem:summability} we prove Theorem~\ref{thm:main}
under the weaker summability assumption \eqref{gather:Q-weak}. The reader who
does not want to follow this slight generalization of the argument may just
neglect all ``$*$'' attached to the norms. We use the following notation:
\begin{displaymath}
  \kappa_N:=\sum_{n=N}^\infty\sup_{\ve\in E}\|Q_\ve^n\|^*\ .
\end{displaymath}

\begin{lem}\label{lem:two-estimates}There is a constant $C>0$ such that, for
  all $\ve\in E$ and all $N\geq0$,
  \begin{enumerate}[a)]
  \item\label{lem:two-estimates-a}
    $|1-\nu_\ve(\vf_0)|\leq C\,\eta_\ve$,
  \item\label{lem:two-estimates-b}
    $\|Q_\ve^N\vf_0\|\leq C\,\kappa_N\big(\|(P_0-P_\ve)(\vf_0)\|_*+|\lambda_0-\lambda_\ve|\big)$.
  \end{enumerate}
\end{lem}
\begin{proof}
  a) As
  $(\Id-\lambda_\ve^{-1}P_\ve)(\lambda_\ve^{-1}P_\ve)^k(\vf_0)=(\Id-\lambda_\ve^{-1}P_\ve)Q_\ve^k(\vf_0)$
  for all $k\geq0$,
  \begin{displaymath}
    \begin{split}
      |1-\nu_\ve(\vf_0)| &=
      \lim_{n\to\infty}\left|\nu_0\left(\vf_0-(\lambda_\ve^{-1}P_\ve)^n(\vf_0)\right)\right|\\
      &\leq
      \sum_{k=0}^\infty\left|\nu_0\left((\Id-\lambda_\ve^{-1}P_\ve)Q_\ve^k(\vf_0)\right)\right|\\
      &=
      \sum_{k=0}^\infty\left|\nu_0\left((\lambda_0^{-1}P_0-\lambda_\ve^{-1}P_\ve)Q_\ve^k(\vf_0)\right)\right|\\
      &\leq |\lambda_0|^{-1}\eta_\ve\sum_{k=0}^\infty\|Q_\ve^k\|^*\|\vf_0\|_*
      +|\lambda_0^{-1}||\lambda_\ve-\lambda_0|\|\nu_0\|\sum_{k=1}^\infty
      \|Q_\ve^k\|^*\|\vf_0\|_*\\
      &= \cO(\eta_\ve)+\cO(\lambda_0-\lambda_\ve)=\cO(\eta_\ve)
    \end{split}
  \end{displaymath}
  where we used \eqref{gather:Q-weak} and \eqref{eq:Lipschitz-at-0} for the last
  estimate.\\ 
  b) For each $N\geq0$ we have
  \begin{displaymath}
    \begin{split}
      \|Q_\ve^N\vf_0\|
      \leq&
      \limsup_{n\to\infty}\|Q_\ve^N(\vf_0-(\lambda_\ve^{-1}P_\ve)^n(\vf_0))\|
      +
      \limsup_{n\to\infty}\|Q_\ve^{N+n}\vf_0\|\\
      \leq&
      \sum_{k=0}^\infty\|Q_\ve^N(\lambda_\ve^{-1}P_\ve)^k(\Id-\lambda_\ve^{-1}P_\ve)(\vf_0)\|+\limsup_{n\to\infty}\kappa_{N+n}\|\vf_0\|_*\\
      \leq&
      |\lambda_0^{-1}|\sum_{k=0}^\infty
      \left(\|Q_\ve^{N+k}(P_0-P_\ve)(\vf_0)\|+|\lambda_0-\lambda_\ve|\|Q_\ve^{N+k+1}\vf_0\|\right)\\
      \leq&
      |\lambda_0^{-1}|\sum_{k=0}^\infty\|Q_\ve^{N+k}\|^*\left(\|(P_0-P_\ve)(\vf_0)\|_*+|\lambda_0-\lambda_\ve|\|\vf_0\|_*\right)\\
      =&
      \cO(\kappa_N)\big(\|(P_0-P_\ve)(\vf_0)\|_*+|\lambda_0-\lambda_\ve|\big)
    \end{split}
  \end{displaymath}
\end{proof}

\begin{proof}[Proof of Theorem~\ref{thm:main}]
  Observe first that by \eqref{eq:basic1}, for each $n>0$,
  \[
  \begin{split}
    &\nu_\ve(\vf_0)\,(\lambda_0-\lambda_\ve)\\
    =&\nu_\ve(\vf_0)\nu_0((P_0-P_\ve)(\vf_\ve))\\
    =&
    \Delta_\ve-\nu_0\left((P_0-P_\ve)(\Id-(\lambda_\ve^{-1}P_\ve)^n)(\vf_0)\right)
    -\nu_0\left((P_0-P_\ve)Q_\ve^n(\vf_0)\right)\\
    =&    
    \Delta_\ve
    -\sum_{k=0}^{n-1}\nu_0\left((P_0-P_\ve)(\lambda_\ve^{-1}P_\ve)^k(\Id-\lambda_\ve^{-1}P_\ve)(\vf_0)\right)
    +\cO(\eta_\ve\|Q_\ve^n\vf_0\|)\\
    =&    
    \Delta_\ve-
    \lambda_0^{-1}\sum_{k=0}^{n-1}\nu_0\left((P_0-P_\ve)(\lambda_\ve^{-1}P_\ve)^k(P_0-P_\ve)(\vf_0)\right)\\
    &+\lambda_0^{-1}(\lambda_0-\lambda_\ve)\sum_{k=1}^{n}\nu_0\left((P_0-P_\ve)(\lambda_\ve^{-1}P_\ve)^{k}(\vf_0)\right)\\
    &+\cO(\kappa_n)\big(|\Delta_\ve|+\eta_\ve|\lambda_0-\lambda_\ve|\big)\hspace*{1.5cm}\text{(by
      Lemma~\ref{lem:two-estimates}\ref{lem:two-estimates-b} and \eqref{gather:small2-strong})}\\
    =&
    \Delta_\ve\left(1-\lambda_0^{-1}\sum_{k=0}^{n-1}\lambda_\ve^{-k}q_{k,\ve}\right)\\
    &+\cO(\eta_\ve)|\lambda_0-\lambda_\ve|\sum_{k=1}^n\left(|\nu_\ve(\vf_0)|\|\vf_\ve\|+\|Q_\ve^k\vf_0\|\right)+\cO(\kappa_n)\big(|\Delta_\ve|+\eta_\ve|\lambda_0-\lambda_\ve|\big)
  \end{split}
\]
where
\begin{equation}
  q_{k,\ve}:=\frac{\nu_0\left((P_0-P_\ve)P_\ve^{k}(P_0-P_\ve)(\vf_0)\right)}{\nu_0\left((P_0-P_\ve)(\vf_0)\right)}.
\end{equation}
Observing Lemma~\ref{lem:two-estimates}\ref{lem:two-estimates-a},
\eqref{gather:phi} and Lemma~\ref{lem:two-estimates}\ref{lem:two-estimates-b},
the error terms can be estimated by
$\cO(\eta_\ve)\,n\,|\lambda_0-\lambda_\ve|+\cO(\kappa_n)|\Delta_\ve|$ so that,
in view of Lemma~\ref{lem:two-estimates}\ref{lem:two-estimates-a}, this
yields, for each $n>0$,
\begin{equation}
  (1+\cO(\eta_\ve))(\lambda_0-\lambda_\ve)(1+n\,\cO(\eta_\ve))
  =
  \Delta_\ve\left(1-\lambda_0^{-1}\sum_{k=0}^{n-1}\lambda_\ve^{-k}q_{k,\ve}\right)
  +\cO(\kappa_n)|\Delta_\ve|.
\end{equation}
If $\Delta_\ve=0$ and $\eta_\ve$ is small, it follows that
$\lambda_\ve=\lambda_0$. Otherwise we assumed in \eqref{eq:main-assumption} that
$q_k=\lim_{\ve\to0}q_{k,\ve}$ exists for each
$k$, and we conclude
\begin{displaymath}
  \lim_{\ve\to0}\frac{\lambda_0-\lambda_\ve}{\Delta_\ve}
  =
  1-\sum_{k=0}^{n-1}\lambda_0^{-(k+1)}q_{k}+\cO(\kappa_n)
\end{displaymath}
for each $n>0$. From this the claim \eqref{eq:main} follows in the limit
$n\to\infty$.
\end{proof}


\begin{thebibliography}{999}
\footnotesize
\bibitem{baladi-2007} Baladi V., \emph{On the susceptibility function of
    piecewise expanding interval maps}, Comm. Math.  Phys. \textbf{275}
  (2007), 839-859
\bibitem{BG} Baladi V., Gouezel S., \emph{Good Banach spaces for piecewise hyperbolic maps via interpolation}, 	arXiv:0711.1960v1
\bibitem{baladi-smania-2008}
  Baladi V. and Smania D., \emph{Linear response formula for piecewise expanding unimodal maps}, Nonlinearity  \textbf{21}  (2008), 677--711
\bibitem{blank-1987} Blank, M. \emph{Stochastic properties of deterministic
    dynamical systems}, Sov. Sci. Rev. C Math./Phys. Vol. \textbf{6} (1987),
  243--271
\bibitem{blank-keller-1997} Blank, M. and Keller, G.
  \emph{Stochastic stability versus localization in one-dimensional chaotic
    dynamical systems}, Nonlinearity \textbf{10} (1997), 81--107
\bibitem{BD} Bunimovich, L.A. and Dettmann, C.P. \emph{Peeping at chaos:
    nondestructive monitoring of chaotic systems by measuring long-time escape
    rates}, Europhys. Lett. EPL 80 (2007), no. 4, Art. 40001, 6 pp.
\bibitem{dellnitz-junge-1999} Dellnitz M. and Junge O., \emph{On the
    Approximation of Complicated Dynamical Behavior}, SIAM Journal on
  Numerical Analysis \textbf{36} (1999), 491--515
\bibitem{DL} Demers M., Liverani C., \emph{Stability of Statistical Properties in Two-dimensional Piecewise Hyperbolic Maps},  Transactions of the American Mathematical Society 360 (2008), 4777-4814
\bibitem{DH} Dettmann C.P. and Howard T.B., \emph{Asymptotic expansions for
    the escape rate of stochastically perturbed dynamical systems},
  arXiv:0805.4570v1
\bibitem{froyland-padberg-2008}
  Froyland G. and Padberg K., \emph{Almost-invariant sets and invariant
    manifolds -- connecting probabilistic and geometric descriptions of
    coherent structures in flows}, Preprint (2008)
\bibitem{froyland-et-al} Froyland G., Padberg K., England M.H. and Treguier
  A.M.,
  \emph{Detection of coherent oceanic
structures via transfer operators}, Physical Review Letters \textbf{98} (2007), 224503
\bibitem{GL} Gouezel S. and Liverani C., \emph{Banach spaces adapted to Anosov
    systems}, Ergodic Theory and Dynamical Systems, 26, 1, 189--217, (2006)
\bibitem{keller-liverani-1999} Keller G. and Liverani C., \emph{Stability of
    the spectrum for transfer operators}, Annali della Scuola Normale
  Superiore di Pisa, Scienze Fisiche e Matematiche (4) \textbf{XXVIII} (1999),
  141--152
\bibitem{keller-liverani-2005} Keller G. and Liverani C., \emph{A spectral gap
    for a one-dimensional lattice of coupled piecewise expanding interval
    maps}, in: Dynamics of Coupled Map Lattices and of Related Spatially
  Extended Systems (Eds.: J.-R. Chazottes, B. Fernandez), Lecture Notes in
  Physics \textbf{671} (2005), pp. 115--151, Springer Verlag
\bibitem{kashyap-2003} Kashyap N.,
\emph{Maximizing the Shannon capacity of constrained systems with two
  constrainets}, SIAM J. Discrete Math. \textbf{17} (2003), 298--319
\bibitem{lasota-yorke-1973} Lasota A. and Yorke, J., \emph{Existence of
    invariant measures for piecewise monotonic transformations},
  Trans. Amer. Math. Soc. \textbf{186} (1973), 481--488
\bibitem{lind-1989} Lind D., \emph{Perturbations of shifts of finite type,}
  SIAM J. Discrete Math. \textbf{2} (1989), 350--365
\bibitem{liverani-maume-2003} Liverani C. and Maume-Deschamps V., \emph{
    Lasota-Yorke maps with holes: conditionally invariant probability measures
    and invariant probability measures on the survivor set}, Annales de
  l'Institut Henri Poincar\'e (B) Probability and Statistics, \textbf{39}
  (2003), 385--412
\bibitem{MDHSch-2006} Meerbach E., Dittmer E., Horenko I. and Schütte Ch.,
  \emph{Multiscale modelling in molecular dynamics: biomolecular conformations
    as metastable states}, in 'Computer Simulations in Condensed Matter: From
  Materials to Chemical Biology. Volume I' Eds.: M. Ferrario, G. Ciccotti, and
  K. Binder Lecture Notes in Physics 703, pp. 475--497, Springer (2006)
\bibitem{MSchF-2005}
  Meerbach D., Sch\"utte, Ch. and Fischer A., \emph{Eigenvalue bounds on
    restrictions of reversible nearly uncoupled Markov chains}, Linear Algebra
  Appl. \textbf{398} (2005), 141--160.
\bibitem{ruelle-1997}
  Ruelle D., \emph{Differentiation of SRB states},
  Comm. Math. Phys. \textbf{187} (1997), 227–-241. See also Comm. Math. Phys.
\textbf{234} (2003), 185–-190
\bibitem{rychlik-1983} Rychlik M., \emph{Bounded variation and invariant
    measures}, Studia Math. \textbf{76} (1983), 69--80
\bibitem{saussol}Saussol, B. \emph{Absolutely continuous invariant measures
    for multidimensional expanding maps}, Israel J. Math. \textbf{116} (2000),
  223--248.
\bibitem{Wael} Bahsoun W., \emph{Rigorous numerical approximation of escape rates},  Nonlinearity  19  (2006),  no. 11, 2529--2542
\end{thebibliography}
\end{document}